\documentclass[11pt,reqno]{article}

\usepackage{amssymb}
\usepackage{latexsym}
\usepackage{amsmath}
\usepackage{enumitem}
\usepackage{hyperref}
\usepackage[noadjust]{cite}
\usepackage{amsthm}
\usepackage{empheq}
\usepackage{bm}
\usepackage[capitalise]{cleveref}
\usepackage{todonotes}

\usepackage{longtable}
\usepackage{multirow,multicol}
\usepackage{caption}
\usepackage{ltablex}\keepXColumns

\newcolumntype{Y}{>{\centering\arraybackslash}X}
\newcolumntype{Z}{>{\scriptsize}Y}
\usepackage{rotating}
\usepackage[T1]{fontenc}
\usepackage[utf8]{inputenc}
\usepackage{authblk} 
\usepackage{setspace}
\numberwithin{equation}{section}
\theoremstyle{definition}
\newtheorem{theorem}{Theorem}[section]
\newtheorem{corollary}[theorem]{Corollary}
\newtheorem{proposition}[theorem]{Proposition}
\newtheorem{definition}[theorem]{Definition}
\newtheorem{example}[theorem]{Example}
\newtheorem{notation}[theorem]{Notation}
\newtheorem{remark}[theorem]{Remark}
\newtheorem{lemma}[theorem]{Lemma}

\newtheorem{claim}{Claim}

\newcommand*{\myproofname}{Proof}
\newenvironment{clproof}[1][\myproofname]{\begin{proof}[#1]}{\end{proof}}
\setlist[enumerate]{labelsep=0.3pc, leftmargin=1.3pc,label=(\alph*),itemsep=.6ex,topsep=0.2ex}
\setlist[enumerate,2]{labelsep=0.3pc, leftmargin=1.6pc,label=(\alph*),itemsep=.6ex,topsep=0.2ex}
\newcommand\qbin[3]{\left[\begin{matrix} #1 \\ #2 \end{matrix} \right]_{#3}}

\newcommand{\rowsp}{\textnormal{rowsp}}
\newcommand{\Fcol}[3]{\mbox{$\F^{{#1}\times{#2}}({#3},{\rm c})$}}
\newcommand{\Frow}[3]{\mbox{$\F^{{#1}\times{#2}}({#3},{\rm r})$}}
\newcommand{\numberset}{\mathbb}
\newcommand{\N}{\numberset{N}}

\newcommand{\R}{\numberset{R}}

\newcommand{\F}{\numberset{F}}

\newcommand{\ini}{\textnormal{ini}}

\newcommand{\mU}{\mathcal{U}}

\newcommand{\mC}{\mathcal{C}}
\newcommand{\mA}{\mathcal{A}}
\newcommand{\mI}{\mathcal{I}}

\newcommand{\rk}{\textnormal{rk}}
\newcommand{\srk}{\textnormal{srk}}

\newcommand{\maxsrk}{\textnormal{maxsrk}}
\newcommand{\maxrk}{\textnormal{maxrk}}
\newcommand{\mB}{\mathcal{B}}
\newcommand{\mP}{\mathcal{P}}

\newcommand{\subspace}[1]{\mbox{$\langle{#1}\rangle$}}
\newcommand{\mediumoplus}{\ensuremath{\vcenter{\hbox{\scalebox{1.5}{$\oplus$}}}}}

\newcommand{\Gaussian}[2]{\genfrac{[}{]}{0pt}{1}{#1}{#2}}

\providecommand{\keywords}[1]
{
  \small	
  \textbf{\textit{Keywords---}} #1
}
\providecommand{\msc}[1]
{
  \small	
  \textbf{\textit{MSC---}} #1
}

\usepackage{array}
\newcolumntype{x}[1]{>{\centering\arraybackslash\hspace{0pt}}p{#1}}
\newcolumntype{y}[1]{>{\centering\arraybackslash\hspace{0pt}}m{#1}}

\title{\textbf{Anticodes in the Sum-Rank Metric}}

\author[1]{Eimear Byrne}
\author[2]{Heide Gluesing-Luerssen\thanks{H. Gluesing-Luerssen was partially supported by the grant \#422479 from the Simons Foundation.}}
\author[3]{Alberto Ravagnani}

\affil[1]{School of Mathematics and Statistics, University College Dublin, Ireland; ebyrne@ucd.ie}
\affil[2]{Department of Mathematics, University of Kentucky, USA; heide.gl@uky.edu}
\affil[3]{Department of Mathematics and Computer Science, Eindhoven University of Technology, The Netherlands; a.ravagnani@tue.nl}

\date{}

\begin{document}

\maketitle

\begin{abstract}\label{sec:Abstract}
We study the structure of anticodes in the sum-rank metric for arbitrary fields and matrix blocks of arbitrary sizes.
Our main result is a complete classification of optimal linear anticodes. 
We also compare the cardinality of the ball in the sum-rank metric with that of an optimal linear anticode, showing that the latter is strictly larger over sufficiently large finite fields. Finally, we give examples of parameters for which the largest anticode is neither a ball nor a 
linear anticode.
\end{abstract}

\\

\msc{ 11T71, 15A30,15A99}

\bigskip

\section{Introduction}

\bigskip

Codes for the sum-rank metric have been proposed in several communication applications (cf.  \cite{MPK19}, \cite{MPK19a}, \cite{NoUF10}, \cite{ShKsch20}). The question of code optimality in the sum-rank has been considered in \cite{us_srk,MP20} and the references therein.

In this note we will study anticodes in the sum-rank metric. More precisely, we consider the space
$\bigoplus_{i=1}^t \F^{n_i \times m_i}$ with an arbitrary field~$\F$, and where the {\em sum-rank} of an element $(X_1,...,X_t)$ 
is the sum of the ranks of the matrices $X_i$. 
The resulting distance function generalizes both the Hamming metric ($n_i=m_i=1$) and the rank metric ($t=1$). 
This hybrid nature is reflected in the classification result of optimal linear anticodes presented later in this note.

In a metric space, an $r$-\textit{anticode} is a set of elements whose pairwise distances are upper bounded by~$r$.
Anticodes naturally arise in the theory of error-correcting codes in two contexts:  
the {\em Code-Anticode Bound} ~\cite{delsarte,ahlswede+} and for \textit{generalized weights}~\cite{ravagnani2016generalized}. 
In both of these topics, optimal anticodes play an important role. 

The Code-Anticode Bound states that a code $C$ of minimum distance $d$ and a $(d-1)$-anticode~$A$, both contained in a suitable 
ambient space $V$, satisfy $|C||A| \leq |V|$. 
This gives an upper bound on $|C|$, and the larger the cardinality of~$A$, the tighter the bound is. Nonlinear anticodes for the Hamming metric were studied in a number of papers (cf. \cite{ahlswedekhachatrian,ahlswede+}), wherein the authors describe codes meeting the Code-Anticode Bound in the case of an optimal anticode as being $d$-{\em diameter perfect}. 
Moreover, the authors establish the size of an optimal anticode as a consequence of their Diametric Theorem.

In \cite{ravagnani2016generalized}, the generalized Hamming and rank weights of a linear code~$C$ are defined in terms 
of the dimensions of intersections of $C$ with optimal linear anticodes. 
This approach leads to duality results for generalized weights~\cite{ravagnani2016generalized} and can be utilized to 
characterize different classes of extremal codes~\cite{mrgiuseppe}.    

A natural question about anticodes is the classification of the optimal linear ones. 
An optimal linear $r$-anticode over a field $\F$ is a linear $r$-anticode whose dimension is maximal over all linear $r$-anticodes over~$\F$.
In the Hamming metric, such linear spaces are isomorphic to $\bigoplus_{i=1}^r \F$ (unless $|\F|=2$); see \cite{ravagnani2016generalized}. 
In the rank metric on $\F^{n \times m}$, the optimal $r$-anticodes are the matrix spaces whose members have column 
spaces contained in a fixed $r$-dimensional subspace $U\subseteq \F^n$ if $n<m$, while for $n=m$ their transposed spaces are also optimal~\cite{Mes85}.

After the preliminaries, we provide in \cref{S-OptLinAnti} a classification of optimal linear anticodes in the sum-rank metric. 
We show that any optimal linear anticode is the direct product of (a) optimal 
anticodes in the rank metric (a certain number of which are necessarily full matrix spaces) and (b) an optimal anticode in the Hamming metric.
Conversely, we also characterize which of these products are indeed optimal linear anticodes.
For the precise formulation see \cref{C-OptAnti}.
    
In \cref{sec:non} we turn to nonlinear anticodes over finite fields and the Code-Anticode Bound.
The latter generalizes the Sphere-Packing Bound by replacing the sphere of radius~$r$ with an optimal $2r$-anticode. 
For example, in the Hamming metric the Singleton bound is sharper than the Hamming bound when the size of 
the sphere of radius~$r$ is exceeded by that of an optimal linear $2r$-anticode (which is the case for sufficiently large field size). 
In Section \ref{sec:non}, we consider this question for sum-rank metric codes and show that, for sufficiently 
large field size, the sum-rank sphere of radius~$r$ is smaller than an optimal linear $2r$-anticode.  

We conclude the paper by providing an example where the largest (possibly non-linear) anticodes in the sum-rank metric are, in general, neither the spheres nor the optimal linear anticodes.

\section{Sum-Rank Metric Anticodes}

Throughout the paper, $\F$ denotes an arbitrary field. For $i \in \N_0=\{0,1,2,\ldots\}$ we define $[i]=\{j \in \N_0 \mid 1 \le j \le i\}$. 
In order to define the ambient space for our sum-rank metric codes we fix positive integers $t$, $n_1, \ldots, n_t$, and $m_1,\ldots,m_t$ that satisfy 
\begin{equation}\label{e-nimi}
     n_i\leq m_i\ \text{ for all }\ i\in[t]\ \text{ and }\ m_1 \ge \cdots \ge m_t.
\end{equation}
We also set 
\begin{equation}\label{e-NM}
    N:=n_1+ \ldots +n_t\ \text{ and }\ M:=m_1+ \ldots +m_t.
\end{equation}
The \textbf{sum-rank metric space} is the product of the $t$ matrix spaces $\F^{n_i \times m_i}$, that is, 
\begin{equation}\label{e-Pi}
 \Pi:=  \Pi(n_1\times m_1,\ldots, n_t\times m_t):=\bigoplus_{i=1}^t \F^{n_i \times m_i}.
\end{equation}
The \textbf{sum-rank} of an element $X=(X_1,...,X_t) \in \Pi$  is
\[
   \srk(X):=\sum_{i=1}^t \rk(X_i).
\]
Note that $\srk$ induces a metric on~$\Pi$ via $(X,Y)\longmapsto \srk(X-Y)$.
In this paper we study anticodes in the sum-rank metric, which are defined as follows.

\begin{definition}
Let $0 \le r \le N$ be an integer. A (\textbf{sum-rank metric}) \textbf{$r$-anticode} is a non-empty subset $A \subseteq \Pi$ such that
$\srk(X-Y) \le r$ for all $X,Y \in A$. We say that $A$ is \textbf{linear} if it is an $\F$-linear subspace of $\Pi$. In that case we write $A \le \Pi$.
\end{definition}

By definition, $A \subseteq \Pi$ is a $0$-anticode if and only if $|A|=1$.
Sum-rank metric anticodes may also be regarded as sets of $(N \times M)$-matrices supported on a particular profile.
More precisely, for any subset $\mP \subseteq [N] \times [M]$, where $N,\,M$ are as in~\eqref{e-NM}, define $\F^{N \times M}[\mP]$ as the space of $(N \times M)$-matrices supported on $\mP$ (i.e., whose nonzero entries all have indices in~$\mP$).
If we define now $n_0=m_0=0$ and
\begin{equation}\label{e-calP}
  N_i=\sum_{j=0}^i n_j,\ M_i=\sum_{j=0}^i m_j,\ \mP^i=[N_{i-1}+1,N_i]\times [M_{i-1}+1,M_i],\ \text{ and }
  \mP=\bigcup_{i=1}^t\mP^i,
\end{equation}
then $\F^{N\times M}[\mP]$ is the space of $(N\times M)$-block-diagonal matrices with $(n_i\times m_i)$-blocks on the diagonal; 
see also Figure~\ref{f1} in Section~\ref{S-OptLinAnti}.
This provides us with an $\F$-linear isomorphism
\begin{equation}\label{e-Psi}
  \psi:\Pi\longrightarrow\F^{N \times M}[\mP],\quad
   (X_1,\ldots,X_t)\longmapsto \begin{pmatrix}X_1 & & \\ &\ddots& \\ & &X_t\end{pmatrix},
\end{equation}
that satisfies $\rk(\psi(X))=\srk(X)$ for all $X\in\Pi$.
Thus~$\psi$ is a linear isometry between the metric spaces $(\Pi,\srk)$ and $(\F^{N \times M}[\mP],\rk)$.
This isometry will be very useful in the next section when characterizing optimal anticodes.

Our first result is an upper bound on the dimension, and thus cardinality, of a \emph{linear} anticode. 
In Section~\ref{sec:non} we will see that this bound is not true for anticodes that are not necessarily linear.

\begin{theorem}\label{T-OptAnti}
Let $0 \le r \le N$ be an integer and set
\[
      \mU_r=\bigg\{(u_1,\ldots,u_t)\in\N_0^t\,\bigg|\, 0\leq u_i\leq n_i \mbox{ for all $i$ and }\sum_{i=1}^t u_i=r\bigg\}.
\]
Then any  linear $r$-anticode $C \le \Pi$ satisfies
\[
             \dim(C) \le \max\bigg\{ \sum_{i=1}^t m_iu_i \,\bigg|\, (u_1,...,u_t) \in \mU_r \bigg\}.
\]
\end{theorem}

\begin{proof}
By \eqref{e-Psi} the space $\psi(C)$ is a linear $r$-anticode in $\F^{N \times M}[\mP]$ with respect to the rank metric. Moreover,
$C$ and $\psi(C)$ have the same cardinality.
 Since $n_i \le m_i$ for all $i \in [t]$, we obtain from \cite[Theorem~46]{GoRa17} that this dimension cannot exceed
 the stated maximum.
\end{proof}

\begin{definition}
We call a linear anticode \textbf{optimal} if it attains the bound in \cref{T-OptAnti} with equality.
\end{definition}

We wish to point out that in \cite{GorlaEtAl20} a weaker notion of anticode optimality is introduced.
The authors also describe explicitly their optimal anticodes. 
The result is similar in nature to ours in the next section.

Clearly, anticodes in the sum-rank metric generalize those in the rank metric (the case $t=1$) and in the Hamming metric 
(the case $n_i=m_i=1$ for all~$i\in[t]$). 
In those instances, the optimal linear anticodes have been described explicitly.
We briefly survey these results.

\begin{notation}\label{notU}
Let $n,m \ge 1$ be integers and let $U \le \F^n$ and $V \le \F^m$ be subspaces.
We denote by $\Fcol{n}{m}{U}$ the space of matrices $X \in \F^{n \times m}$ whose column space is contained in $U$. Similarly, we let $\Frow{n}{m}{V}$ denote space of matrices $X \in \F^{n \times m}$ whose row space is contained in $V$. 
\end{notation}

It is not difficult to see that the spaces $\Fcol{n}{m}{U}$ and $\Frow{n}{m}{V}$ in Notation~\ref{notU} have dimension
$m\dim(U)$ and $n\dim(V)$, respectively.

The following result has been proven by Meshulam in~\cite{Mes85} for square matrices, 
but it is easy to verify that it is equally true for general rectangular matrices.

\begin{theorem}[\text{\cite[Thm.~3]{Mes85}}]\label{T-Meshulam}
Let $m \ge n \ge 1$ and $0 \le r \le n$ be integers.  
Let $A\leq\F^{n\times m}$ be an optimal linear $r$-anticode.
Then there exists an $r$-dimensional subspace $U\leq\F^n$  such that $A=\Fcol{n}{m}{U}$ or, if $n=m$,
$A=\Frow{n}{m}{U}$.
\end{theorem}

The next result from~\cite{ravagnani2016generalized} describes optimal linear anticodes in the Hamming metric over fields with at least 3 elements.
In that paper the result has been stated for finite fields, 
but the proof (and the proof of the accompanying lemma) shows that it is true for arbitrary fields.

\begin{theorem}[\mbox{\cite[Prop.~9]{ravagnani2016generalized}}]\label{T-OptAntiHamm}
Suppose that $|\F|\geq3$ and let $A\leq\F^t$ be an optimal $r$-anticode in the Hamming metric. 
Then $\dim(A)=r$ (see also Theorem~\ref{T-OptAnti}) and there exists a subset $\mI\subseteq[t]$ such that $|\mI|=r$ and
$A=A_1\oplus\ldots\oplus A_t$, where $A_i=\F$ if $i\in\mI$ and $A_i=\{0\}$ otherwise.
Note that, trivially, each~$A_i$ is an optimal anticode in~$\F$.
\end{theorem}

The simple example $A=\subspace{(1,0,1),(0,1,1)}$ shows that the previous result is not true for the binary field~$\F=\F_2$.

\section{Optimal Linear Anticodes}\label{S-OptLinAnti}
In this section we give a complete classification of optimal linear anticodes $A \le \Pi$ over any field~$\F$. 
We prove that every such anticode is the direct product of optimal rank-metric anticodes and an optimal Hamming-metric anticode. 
More precisely, the following holds.
Recall the notation from~\eqref{e-nimi}--\eqref{e-NM}.

\begin{theorem}\label{T-OptAntiPlain} 
Suppose $m_1\geq\ldots\geq m_s>m_{s+1}=\ldots=m_t=1$, where  the case $s=t$, thus $m_i>1$ for all~$i$, is allowed.
Let $A\leq\Pi$ be an optimal linear anticode.
Then $A=A_1\oplus\ldots\oplus A_s\oplus A^{\sf h}$, where~$A_i$ is an optimal linear anticode in $\F^{n_i\times m_i}$ for $i\in[s]$ and
$A^{\sf h}\leq\F^{1\times (t-s)}$ is an optimal linear anticode with respect to the Hamming metric.
In particular, if $|\F|\geq3$ then $A^{\sf h}=A^{\sf h}_1\oplus\ldots\oplus A^{\sf h}_{t-s}$, where each $A^{\sf h}_i\in\{\F,\{0\}\}$.
\end{theorem}

Before we turn to the proof we present a numerical lemma that describes the maximum in Theorem~\ref{T-OptAnti}.

\begin{lemma}\label{L-Maximize}
Let $0\leq r< N$ and set $K:=\max\{\sum_{i=1}^t m_i u_i\mid (u_1,\ldots,u_t)\in\mU_r\}$.
Let $j\in[t]$ and $\delta\in\{0,\ldots,n_j-1\}$ 
be the unique integers such that $r=\sum_{i=1}^{j-1}n_i+\delta$.
Then the following hold.
\begin{enumerate}
\item $K=\sum_{i=1}^{j-1}m_in_i+m_j\delta$.
\item Let $\ell,\ell'$ be the unique integers such that  
         \[
             m_1\geq \ldots\geq m_{\ell-1}> m_{\ell}=\ldots=m_j=\ldots=m_{\ell'}>m_{\ell'+1}\geq\ldots\geq m_t
         \]
         and let $(u_1,\ldots,u_t)\in\mU_r$. Then $\sum_{i=1}^t m_i u_i=K$ if and only if
      \begin{enumerate}[label=(\roman*)]
      \item $u_i=n_i$ for $i<\ell$,
      \item $\sum_{i=\ell}^{\ell'}u_i=\sum_{i=\ell}^{j-1}n_i+\delta$,
       \item $u_i=0$ for $i>\ell'$.
      \end{enumerate}
\end{enumerate}
\end{lemma}

\begin{proof} Part (a) is \cite[Lemma 3.12]{us_srk}. Let us prove part (b).
It is easy to see that if $(u_1,\ldots,u_t)\in \mU_r$ satisfies (i)--(iii), then $\sum_{i=1}^t m_iu_i=K$.
Thus, let $(u_1,\ldots,u_t)\in\mU_r$ be such that $\sum_{i=1}^tm_iu_i=K$.
We first show~(i).
Write $u_i=n_i-\alpha_i$ for all $i\in[t]$.
Assume by contradiction that $\sum_{i=1}^{\ell-1}\alpha_i>0$.
This means in particular that $\ell\geq2$ and thus the set $\{m_1,\ldots,m_{\ell-1}\}$ is not empty.
Since each element in this set satisfies
$m_i>m_\ell=m_j$,
we obtain the following strict inequality
\allowdisplaybreaks
\begin{align*}
  \sum_{i=1}^{j-1}m_iu_i&=\sum_{i=1}^{j-1}m_in_i-\sum_{i=1}^{j-1}m_i\alpha_i
  \\ &=K-m_j\delta-\sum_{i=1}^{\ell-1}m_i\alpha_i-m_j\sum_{i=\ell}^{j-1}\alpha_i\\
   &<K-m_j(\delta+\sum_{i=1}^{j-1}\alpha_i).
\end{align*}
As a consequence $m_j\sum_{i=j}^t u_i\geq\sum_{i=j}^tm_iu_i=K-\sum_{i=1}^{j-1}m_iu_i>m_j(\delta+\sum_{i=1}^{j-1}\alpha_i)$ and thus
$\sum_{i=j}^t u_i>\delta+\sum_{i=1}^{j-1}\alpha_i$, which in turn yields
\[
  r=\sum_{i=1}^t u_i=\sum_{i=1}^{j-1}(n_i-\alpha_i)+\sum_{i=j}^tu_i>\sum_{i=1}^{j-1}(n_i-\alpha_i)+\delta+\sum_{i=1}^{j-1}\alpha_i=r,
\]
which is a contradiction. This proves~(i).
Next, part~(i) along with $\sum_{i=1}^t u_i=\sum_{i=1}^{j-1}n_i+\delta$ implies
\begin{equation}\label{e-Partiii}
     \sum_{i=\ell}^t u_i=\sum_{i=\ell}^{j-1}n_i+\delta.
\end{equation}
Now the assumption $K=\sum_{i=1}^t m_iu_i$ together with \eqref{e-Partiii} and~(a) yields
\[
   K=\sum_{i=1}^tm_iu_i\leq\sum_{i=1}^{\ell-1}m_in_i+m_\ell\sum_{i=\ell}^tu_i=\sum_{i=1}^{\ell-1}m_in_i
   +m_{\ell}\sum_{i=\ell}^{j-1}n_i+m_{\ell}\delta= K.
\]
As a consequence, the inequality is actually an equality, and this means $\sum_{i=\ell}^t(m_\ell-m_i)u_i=0$.
Since all terms in this sum are non-negative and $m_\ell>m_i$ for $i>\ell'$, this implies $u_i=0$ for $i>\ell'$. This proves~(iii) and now~(ii) follows from \eqref{e-Partiii}.
\end{proof}

In the proof of \cref{T-OptAntiPlain} we will use
the isometry~$\psi$ between the metric spaces $(\Pi,\srk)$ and $(\F^{N\times M}[\mP],\rk)$ in~\eqref{e-Psi} and 
where~$\mP=\cup_{i=1}^t\mP^i$ is as in \eqref{e-calP}.
Note that~$\mP^i$ is the index set of the entries in the $i$-th block on the diagonal of the matrices in $\F^{N\times M}[\mP]$; see also Figure~\ref{f1}.
We need additional terminology.

\begin{definition}\label{D-calP}
\begin{enumerate}
\item For $i\in[t]$, $a\in[N_{i-1}+1,N_i]$ and $b\in[M_{i-1}+1,M_i]$ define
          \[
                L_{a,\boldsymbol{\cdot}}^i:=\left\{(a,y) : M_{i-1}+1 \leq y \leq M_i \right\} \text{ and } 
                L_{\boldsymbol{\cdot},b}^i:=\left\{(x,b) : N_{i-1}+1 \leq  x \leq N_i \right\}.
          \]
          Hence $L_{a,\boldsymbol{\cdot}}^i$ is the horizontal line through position $(a,M_i)$ in the $i$-th rectangular block 
          and~$L_{\boldsymbol{\cdot},b}^i$ is the vertical line through position $(N_i,b)$ in the same block (more precisely, these are just line segments).
\item Set $L(\mP):=\{L_{a,\boldsymbol{\cdot}}^i\mid i\in[t],a\in[N_{i-1}+1,N_i]\}\cup\{L_{\boldsymbol{\cdot},b}^i\mid i\in[t],b\in[M_{i-1}+1,M_i]\}$, 
         that is, $L(\mP)$ is the set of all lines in $\mP$.
\end{enumerate}
\end{definition}

\begin{figure}
\begin{center}
\begin{tikzpicture}[scale=0.5]
   \draw (-2.5,5) node (x0) [label=center: \mbox{$[N]\times[M]=$}] {} ;
   \draw (0,0) rectangle (11,9);
   \filldraw[black!13!white,draw=black] (9,0) rectangle (11,2);
   \filldraw[black!13!white,draw=black] (0,6) rectangle (4,9);
   \filldraw[black!13!white,draw=black] (4,4) rectangle (7,6);
   \draw (8,3) node (x1) [label=center: $\ddots$] {} ;
   \draw[dashed] (5,4) -- (5,6);
  \draw[dashed] (4,4.5) -- (7,4.5);
  \draw (5,6.4) node (x2) [label=center: \scriptsize{$L_{\boldsymbol{\cdot},b}^2$}] {} ;
  \draw (7.6,4.5) node (x3) [label=center: \scriptsize{$L_{a,\boldsymbol{\cdot}}^2$}] {} ;
  \draw (2,7.5) node (x4) [label=center: \small{$\mP^1$}] {} ;
  \draw (10,1) node (x5) [label=center: \small{$\mP^t$}] {} ;
  \draw (6,5) node (x5) [label=center: \small{$\mP^2$}] {} ;
\end{tikzpicture}
\caption{}\label{f1}
\end{center}
\end{figure}

As usual, for a matrix $B\in\F^{N\times M}$, we write~$B_{i,j}$ for its entry at position  $(i,j)$, and for a vector~$v$ we write~$v_j$ for
its $j$-th entry.

\begin{definition}\label{D-iniB}
For a nonzero $B\in\F^{N\times M}$ we set $\ini(B):=\min\{(i,j)\mid B_{i,j}\neq0\}$, where the minimum is taken with respect to the lexicographic order on  $[N]\times [M]$, i.e.,
$(i,j)<(i',j')$ iff $[i<i']$ or $[i=i'$ and $j<j']$.
We call $\ini(B)$ the \textbf{position of the initial entry} of~$B$.
For any set of matrices $\mB \subseteq \Pi$, define $\ini(\mB)=\{\ini(B)\mid B\in\mB, \, B \neq 0\}$.
Finally, let 
\[
     \rho(\mB):=\min\big\{ |S| \, : \, S \subseteq L(\mP), \, \ini(\mB) \subseteq \cup_{\ell\in S}\ell \big\},
\]
that is, $\rho(\mB)$ is the minimum number of lines in~$\mP$ whose union contains $\ini(\mB)$.
\end{definition}

We will need the following result of Meshulam. The proof in~\cite{Mes85} is for square matrices but the same argument remains valid for rectangular ones.

\begin{theorem}[\text{\cite[Thm.~1]{Mes85}}] \label{T-M}
Let $\mB\subseteq\F^{N\times M}$ be a non-empty subset. Then the span of $\mB$ contains a
matrix of rank at least $\rho(\mB)$. 
\end{theorem}

The main step for proving \cref{T-OptAntiPlain} is the following result for which we need the projections
\begin{equation}\label{e-Proj}
\left.\begin{split}
  p_1:\;&\Pi\longrightarrow \F^{n_1\times m_1},&\ (X_1,\ldots,X_t)&\longmapsto X_1,\\
   p_{2,...,t}:\;&\Pi\longrightarrow \mediumoplus_{i=2}^t\F^{n_i\times m_i},& \ (X_1,\ldots,X_t)&\longmapsto \ (X_2,\ldots,X_t).
\end{split}\qquad\right\}
\end{equation}

\begin{theorem}\label{tecn}
Suppose $t \ge 2$ and $m_1 \ge 2$.
Let $A \le \Pi$ be an optimal linear anticode. Then
$A=p_1(A) \oplus p_{2,...,t}(A)$.
Moreover, the spaces $p_1(A)$ and $p_{2,...,t}(A)$ are optimal linear anticodes
in $\F^{n_1 \times m_1}$ and $\mediumoplus_{i=2}^t\F^{n_i \times m_i}$,
respectively.
\end{theorem}

\begin{proof}
Let $r=\maxsrk(A)$, i.e.~$A$ is an $r$-anticode.
For the first part of the proof (through Claim~\ref{cl4}) we will identify~$\Pi$ with the space of block diagonal matrices 
in $\F^{N\times M}[\mP]$ as in~\eqref{e-Psi}.
This will allow us to make use of \cref{T-M}.
Recall that the sum-rank in~$\Pi$ equals the rank in $\F^{N\times M}[\mP]$.
We also need the $\F$-isomorphism
\[
    \varphi: \F^{N \times M} \longrightarrow \F^{NM},\quad (z_{i,j}) \longmapsto (z_{1,1},...,z_{1,M},...,z_{N,1},...,z_{N,M}),
\]
which simply lists the entries of the matrix $(z_{i,j})$ according to the lexicographic order.
In particular, $\varphi(z)_{(i-1)M+j}=z_{i,j}$.

Let $k:=\dim(A)$ and $\tilde{\mB}=\{\tilde{B}_1,\ldots\tilde{B}_k\}$ be a basis of $A$.
Consider the matrix in $\F^{k\times NM}$ with rows
$\varphi(\tilde{B}_1),\ldots,$ $\varphi(\tilde{B}_k)$ and let
\[
   Z=\begin{pmatrix}Z_1\\ \vdots\\ Z_k\end{pmatrix} \in\F^{k\times NM}
\]
be its reduced row echelon form. Set
\begin{equation}\label{e-Bj}
         B_j:=\varphi^{-1}(Z_j)\ \text{ for }\ j\in[k].
\end{equation}
Then $B_j\in\F^{N\times M}[\mP]$ for all $j$ and $\mB=\{B_1,\ldots,B_k\}$ is  a basis of~$A$.
For each $i \in[t]$ set $\mB^i=\{ B \in \mB : \ini(B) \in \mP^i \}$.
Clearly $\mB$ is the disjoint union of $\mB^1,...,\mB^t$.
We now prove a series of properties.

\begin{claim} \label{cl1}
$|\ini(\mB)| = k$, i.e., the positions of the initial entries of the matrices in $\mB$ are distinct.
\end{claim}

\begin{clproof}
This follows from the definition of the map~$\varphi$, which guarantees that the $j$th pivot index of~$Z$ (i.e., the index of the first nonzero entry of $Z_j$) corresponds to $\ini(B_j)$.
\end{clproof}

Another immediate consequence of the reduced row echelon form is the following.

\begin{claim} \label{cl2}
For each $B\in\mB$ and $(a,b)\in\ini(\mB)\setminus\ini(B)$ we have $B_{a,b}=0$.
\end{claim}

\begin{claim} \label{cl3}
There exist $r_1,\ldots,r_t\in\N_0$ such that 
 \begin{equation}\label{e-ri}
    0\leq r_i\leq n_i,\quad \sum_{i=1}^{t} r_i = r,\quad \sum_{i=1}^{t} r_im_i = k,\quad  \ini(\mB^i) \text{ is a union of $r_i$ lines in $\mP^i$}.
 \end{equation}
\end{claim}

\begin{clproof}
Consider $\rho(\mB)$ from \cref{D-iniB}.
Thanks to Theorem~\ref{T-M} we have $\rho(\mB) \leq r$ since the rank of any element of $A$ is upper bounded by~$r$.
In order to show equality, let $S\subseteq L(\mP)$ be a set of lines in~$\mP$ of minimum size such that $\ini(\mB)\subseteq \cup_{\ell\in S}\ell$.
Let~$r_i$ be the number of lines in $S$ that are in~$\mP^i$. Thus $\rho(\mB)=|S|=\sum_{i=1}^t r_i$ and $0\leq r_i\leq n_i$.
Since each line in~$\mP^i$ contains at most~$m_i$ points, we conclude $\sum_{i=1}^t r_im_i\geq k$.
Furthermore, we know already that $\sum_{i=1}^t r_i\leq r$, and therefore the optimality of~$A$ and \cref{T-OptAnti} tell us that $\sum_{i=1}^t r_im_i=k$ and $\sum_{i=1}^t r_i=r$.
Hence $\rho(\mB)=r$.
Now \eqref{e-ri} follows from Claim~\ref{cl1}.
\end{clproof}

\begin{claim} \label{cl4}
Without loss of generality, we may assume that the $r_i$ lines in~$\mP^i$ are all horizontal for all $1 \le i \le t$,
and that they are the topmost lines in every block.
\end{claim}

\begin{clproof}
From~\eqref{e-ri} along with $\ini(\mB)=k=\sum_{i=1}^t r_im_i$ we conclude that the~$r_i$ lines in~$\mP^i$ covering $\ini(\mB^i)$ must be disjoint.
Moreover, if $n_i<m_i$ they must be horizontal, hence parallel.
If $n_i=m_i$, they are either all vertical or all horizontal.
Since the assumption and conclusion of the theorem are invariant under row permutations and column permutations of the matrices in a given block,
we may assume without loss of generality that these lines are the topmost lines or the leftmost lines (if they are vertical).
In addition, the theorem is invariant under transposition of individual blocks and thus we may assume that all lines are horizontal. 
Note that all the described operations do not change the numbers $r_1,...,r_t$.
\end{clproof}

For the rest of the proof we return to presenting the elements of the anticode~$A$ as matrix tuples in~$\Pi$ (which saves us from displaying large block diagonal matrices).
In particular, we identify~$B_j$~in~\eqref{e-Bj} with its matrix tuple.
We can now describe these basis elements more precisely.

In each $\mP^i$ the union of the $r_i$ horizontal lines from~\eqref{e-ri} form a rectangle $L_i$ of size $r_i\times m_i$.
Hence $L_i=\ini(\mB^i)$.
Under the map~$\varphi$ the union $\bigcup_{i=1}^t L_i$ corresponds to the set of pivot indices of the matrix~$Z$.

It will be convenient to use a different indexing for the matrices in the basis~$\mB=\{B_1,\ldots,B_k\}$.
For $\ell\in\{1,\ldots,t\}$, $i\in\{1,\ldots,r_\ell\}$ and $j\in\{1,\ldots,m_\ell\}$, let
$B^\ell_{i,j}:=B_{\lambda}$, where $\lambda=\sum_{x=1}^{\ell-1}r_xm_x+m_\ell(i-1)+j$.
This simply amounts to, see also Claims~\ref{cl2} and~\ref{cl4}, 
\begin{equation}\label{e-Bellij}
 B^\ell_{i,j}=\left(0,\ldots,0,\begin{pmatrix}\;E^\ell_{i,j}\;\\[.4ex]\hline *\end{pmatrix},
         \begin{pmatrix}\;0_{r_{\ell+1}\times m_{\ell+1}}\;\\[.4ex]\hline *\end{pmatrix},\ldots,
         \begin{pmatrix}\;0_{r_t\times m_t}\;\\[.4ex]\hline *\end{pmatrix}\right),
\end{equation}
where the first nonzero block is at position~$\ell$ and $E^\ell_{i,j}\in\F^{r_\ell\times m_\ell}$ is the standard basis matrix with entry~$1$ 
at position $(i,j)$.
Clearly, $B^\ell_{i,j}$ does not exist if $r_\ell=0$, in which case no matrix tuple in~$A$ starts in the $\ell$-th block.
We do not need to specify the lower parts of the matrix tuples in~\eqref{e-Bellij}.
In Claim~\ref{cl5} however, we will see that for $\ell=1$ the second to last block are zero.
By construction,
\begin{equation}\label{e-mBell}
   \mB^\ell=\subspace{B^\ell_{i,j}\mid 1\leq i\leq r_\ell,\,1\leq j\leq m_\ell}.
\end{equation}

\begin{claim} \label{cl5}
Every $B\in\mB^{(1)}$ is of the form $B=(X_1,0,\ldots,0)$.
As a consequence,  $A=p_1(A) \oplus p_{2,...,t}(A)$, where~$p_1$ and $p_{2,...,t}$ are as in~\eqref{e-Proj}.
\end{claim}

\begin{clproof}
The result is clear if $r_1=0$. We henceforth assume $r_1 \ge 1$ and prove that
every $B\in\mB^1$ is of the form $B=(X_1,0,\ldots,0)$.
The second part of the claim then follows because 
$A \subseteq p_1(A) \oplus p_{2,...,t}(A)$ and~$A$ is generated by the
matrices in $\mB=\mB^1\cup\bigcup_{i=2}^t\mB^i$.

Thanks to~\eqref{e-mBell} it suffices to show the statement for the basis matrices $B^1_{i,j}$.
Since the claim is invariant under permutation of the rows $1,...,r_1$ in block~1 and
permutation of the columns in block~1, it suffices to proves the claim for $B^1_{1,1}$.
Write $B^1_{1,1}=(X_1,\ldots,X_t)$ and fix $\ell\geq 2$.
We show that $X_\ell=0$.
Thanks to~\eqref{e-Bellij} it remains to show that the last $n_\ell-r_\ell$ rows of~$X_\ell$ are zero.
If $r_\ell=n_\ell$, we are done. Thus suppose $r_\ell<n_\ell$, which implies $r_\ell<m_\ell$ as well.

The remainder of our argument is inspired by the proof of~\cite[Claim 2]{Mes85}.
Note that our claim is also invariant under permutation of rows $r_\ell+1,...,n_\ell$ in block~$\ell$ and permutation
of the columns in block~$\ell$, and ttherefore it suffices to show that $(X_\ell)_{r_\ell+1,r_\ell+1}=0$.

For any $\nu,i,j$ we define $e^\nu_{i,j}$ to be the entry at position $(r_\ell+1,r_\ell+1)$ in the $\ell$-th block of~$B^{\nu}_{i,j}$.
Precisely, if $B^\nu_{i,j}=(M_1,\ldots,M_t)$, then $e^{\nu}_{i,j}=(M_\ell)_{(r_\ell+1,r_\ell+1)}$. We distinguish two cases.

\noindent\underline{Case I}: $r_1 \ge 2$. 
Define 
\[
  Z:=B^1_{2,1}+B^1_{1,2}+\sum_{x=3}^{r_1}B^1_{x,x}+\sum_{\nu=2}^t\sum_{x=1}^{r_\nu}B^\nu_{x,x}\quad \text{ and }\quad T:=Z+B^1_{1,1}.
\]
Then $Z,\,T\in A$.
We can describe these matrices more explicitly. 
Set $\tilde{I}_{r_\nu}=(I_{r_\nu}\mid 0)\in\F^{r_\nu\times m_\nu}$, $\tilde{J}:=(J\mid 0)\in\F^{r_1\times m_1}$ and
$\hat{J}:=(J'\mid 0)\in\F^{r_1\times m_1}$, where
\[
    J=\begin{pmatrix}0\!&\!1\!&\! \!&\! \!&\! \\1\!&\!0\!&\! \!&\! \!&\!\\ \!&\! \!&\!1\!&\! \!&\!\\ \!&\!\!&\!\!&\!\ddots\!&\! \\ \!&\! \!&\! \!&\! \!&\!1\end{pmatrix},\quad 
    J'=\begin{pmatrix}1\!&\!1\!&\! \!&\! \!&\! \\1\!&\!0\!&\! \!&\! \!&\!\\ \!&\! \!&\!1\!&\! \!&\!\\ \!&\!\!&\!\!&\!\ddots\!&\! \\ \!&\! \!&\! \!&\! \!&\!1\end{pmatrix}
    \in\F^{r_1\times r_1}.
\]
Then~$Z$ and~$T$ are of the form
\[
   Z=\left(\begin{pmatrix}\tilde{J} \\ \hline Z^1_2\end{pmatrix},\,\begin{pmatrix}\tilde{I}_{r_2} \\ \hline Z^2_2 \end{pmatrix},\ldots,\,\begin{pmatrix}\tilde{I}_{r_t} \\ \hline Z^t_2 \end{pmatrix}\right)
   \quad \text{ and } \quad
   T=\left(\begin{pmatrix}J' \\ \hline T^1_2\end{pmatrix},\,\begin{pmatrix}\tilde{I}_{r_2} \\ \hline T^2_2 \end{pmatrix},\ldots,\,\begin{pmatrix}\tilde{I}_{r_t} \\ \hline T^t_2\end{pmatrix}\right)
\]
for some matrices $Z^{\nu}_2$ and $T^{\nu}_2$ of fitting sizes.
Since the sum-rank of the upper blocks of both~$Z$ and~$T$ equals $\sum_{\nu=1}^t r_{\nu}=r=\maxrk(A)$,  the lower blocks cannot 
contribute anything to the sum rank of the matrix tuples.
In particular, the row spaces of these matrices satisfy
$\rowsp(Z^\nu_2)\subseteq\rowsp(\tilde{I}_{r_\nu})$ and $\rowsp(T^\nu_2)\subseteq\rowsp(\tilde{I}_{r_\nu})$ for all $\nu\geq2$.
This implies in particular that  $(Z^\ell_2)_{1,r_\ell+1}=0=(T^\ell_2)_{1,r_\ell+1}$.
But these entries are given by 
\[
     \alpha:=e^1_{1,2}+e^1_{2,1}+\sum_{x=3}^{r_1}e^1_{x,x}+\sum_{\nu=2}^t\sum_{x=1}^{r_\nu}e^\nu_{x,x}\ \text{ and }\ \alpha+e^1_{1,1},
\]
respectively.
This leads to $e^1_{1,1}=0$, as desired.

\noindent\underline{Case II}: $r_1 =1$. \  
Since by assumption $m_1\geq2$ we may consider the matrix tuples
\[
  Z:=B^1_{1,2}+\sum_{\nu=2}^t\sum_{x=1}^{r_\nu}B^\nu_{x,x}\ \text{ and }\ T:=Z+B^1_{1,1}.
\]
Set again $\tilde{I}_{r_\nu}=(I_{r_\nu}\mid 0)\in\F^{r_\nu\times m_\nu}$.
Then~$Z$ and~$T$ are of the form
\[
   Z=\left(\begin{pmatrix}0\,1\,0\ldots0\\ \hline Z^1_2\end{pmatrix},\,\begin{pmatrix}\tilde{I}_{r_2} \\ \hline Z^2_2\end{pmatrix},\ldots,\,\begin{pmatrix}\tilde{I}_{r_t} \\ \hline Z^t_2\end{pmatrix}\right)
   \ \text{ and }\
   T=\left(\begin{pmatrix} 1\,1\,0\ldots0\\ \hline T^2_2\end{pmatrix},\,\begin{pmatrix}\tilde{I}_{r_2} \\ \hline T^2_2 \end{pmatrix},\ldots,\,\begin{pmatrix}\tilde{I}_{r_t} \\ \hline T^t_2\end{pmatrix}\right).
\]
As in the previous case, the sum-rank of the upper blocks of both~$Z$ and~$T$ equals $\sum_{\nu=1}^t r_{\nu}=r=\maxrk(A)$. Thus
$\rowsp(Z^\nu_2)\subseteq\rowsp(\tilde{I}_{r_\nu})$ and $\rowsp(T^\nu_2)\subseteq\rowsp(\tilde{I}_{r_\nu})$ for all $\nu\geq2$.
This implies  $(Z^\ell_2)_{1,r_\ell+1}=0=(T^\ell_2)_{1,r_\ell+1}$.
These entries are
\[
     \alpha:=e^1_{1,2}+\sum_{\nu=2}^t\sum_{x=1}^{r_\nu}e^\nu_{x,x}\ \text{ and }\ \alpha+e^1_{1,1},
\]
respectively, and we arrive again at $e^1_{1,1}=0$.
\end{clproof}

\begin{claim} \label{cl6}
The spaces $p_1(A)$ and $p_{2,...,t}(A)$ in Claim~\ref{cl5} are optimal anticodes
in $\F^{n_1 \times m_1}$ and $\mediumoplus_{i=2}^t\F^{n_i \times m_i}$
respectively.
\end{claim}

\begin{clproof}
By the previous claim $A = p_1(A) \oplus p_{2,...,t}(A)$.
Thus thanks to~\eqref{e-ri} and~\eqref{e-Bellij} we have
$r_1+ \cdots +r_t=\maxsrk(A)=\maxrk(p_1(A)) + \maxsrk(p_{2,...,t}(A))$ as well as $\maxrk(p_1(A)) \ge r_1$ and $\maxsrk(p_{2,...,t}(A)) \ge r_2 + \cdots +r_t$.
Hence we  have  equality for both projections.
Moreover, since $p_{2,\ldots,t}(\mB^1)=0$ by Claim~E, we have
$p_1(A)=p_1(\subspace{\mB^1})$ and
$p_{2,...,t}(A)=p_{2,\ldots,t}(\subspace{\mB^2 \cup \cdots \cup \mB^t})$ and thus these spaces have dimensions
$r_1m_1$ and $r_2m_2 + \cdots + r_tm_t$, respectively.
All of this together with the fact that~$A$ is an optimal $r$-anticode tells us that $p_1(A)$ and $p_{2,...,t}(A)$ are optimal anticodes.
\end{clproof}

All of this establishes the theorem.
\end{proof}

\begin{proof}[Proof of Theorem~\ref{T-OptAntiPlain}] The result now follows by combining Theorem~\ref{tecn}, induction, 
and Theorem~\ref{T-OptAntiHamm}.
\end{proof}

\begin{remark}
Not every anticode $A_1\oplus\ldots\oplus A_s\times A^{\sf h}$ as specified in Theorem~\ref{T-OptAntiPlain} is an optimal linear anticode.
Take for instance $\Pi=\F^{3\times3}\oplus\F^{2\times2}$ and $r=4$.
Then $A:=\Fcol{3}{3}{U_1}\times\Fcol{2}{2}{U_2}$ is an $r$-anticode whenever $\dim(U_1)+\dim(U_2)=4$, and each factor is an optimal linear anticode in the rank metric.
However, for $\dim(U_1)=2=\dim(U_2)$ the anticode~$A$ has dimension~$10$, whereas $\dim(A)=11$ for
$\dim(U_1)=3$ and $\dim(U_2)=1$.
The latter choice turns~$A$ into an optimal linear anticode by \cref{T-OptAnti}.
\end{remark}

The particular choice of factors in \cref{T-OptAntiPlain} that lead to optimal linear anticodes can be made precise with the aid of \cref{L-Maximize}.
This leads to the following complete classification of optimal linear anticodes.
Define 
\[
   \mA_r(\Pi)=\{A\leq \Pi\mid A \text{ is an optimal linear $r$-anticode}\}.
\]

\begin{corollary}[\textbf{Classification of Optimal linear Anticodes}]\label{C-OptAnti}
Let $0\leq r< N$ and let $j\in[t]$ and $0\leq\delta<n_j$ be the unique integers such that
$r=\sum_{i=1}^{j-1}n_i+\delta$.
\begin{enumerate}
\item Suppose $m_j>1$. Let
        $m_1\geq\ldots\geq m_{\ell-1}> m_{\ell}=\ldots=m_j=\!\ldots\!=m_{\ell'}>m_{\ell'+1}\geq\ldots\geq m_t$.
        Furthermore, let $j\leq s\leq t$ be such that $m_i=1$ iff $i>s$. 
        Then
        \[
           \mA_r(\Pi)=\bigg\{\bigoplus_{i=1}^{\ell-1}\F^{n_i\times m_i}\oplus\bigoplus_{i=\ell}^{\ell'} A_i\bigg|
            \begin{array}{l} 
                A_i\leq \F^{n_i\times m_i}\text{ is an optimal $r_i$-anticode for some} \\[.7ex]     
                0\leq r_i\leq n_i\text{ such that } \sum_{i=\ell}^{\ell'}r_i=\sum_{i=\ell}^{j-1}n_i+\delta  
              \end{array}\!\!\bigg\}.
         \]
         \
         \\[-2.5ex]
\item Suppose $m_j=1$ and let $m_1\geq\!\ldots\!\geq m_{\ell-1}> m_{\ell}=\!\ldots\!=m_j=\!\ldots\!= m_t=1$.
        Then
        \[
           \mA_r(\Pi)=\bigg\{\bigoplus_{i=1}^{\ell-1}\F^{n_i\times m_i}\oplus A^{\sf h}\,\bigg|\,
               A^{\sf h}\leq\F^{t-\ell+1}\text{ is an optimal $(j-\ell)$-anticode}  
               \bigg\}.
         \]
\end{enumerate}
\end{corollary}

\begin{proof}
(a) 
``$\supseteq$'' \; If~$A$ is as in the set on the right hand side, then~$A$ is a $\rho$-anticode for
$\rho=\sum_{i=1}^{\ell-1}n_i+\sum_{i=\ell}^{\ell'}r_i=\sum_{i=1}^{\ell-1}n_i+\sum_{i=\ell}^{j-1}n_i+\delta=r$.
The optimality follows from the fact that $\dim(A)=\sum_{i=1}^{\ell-1}n_im_i+\sum_{i=\ell}^{\ell'}\dim(A_i)=\sum_{i=1}^{\ell-1}n_im_i+\sum_{i=\ell}^{\ell'}r_im_i$,
along with \cref{L-Maximize}.
\\
``$\subseteq$'' \; Let $A\in\mA_r(\Pi)$. By \cref{T-OptAntiPlain}  $A=A_1\oplus\ldots\oplus A_s\oplus A^{\sf h}$, where~$A_i$ is an optimal $r_i$-anticode in 
$\F^{n_i\times m_i}$ for some $0\leq r_i\leq n_i$ and $A^{\sf h}$ is an optimal $r_{s+1}$-anticode in $\F^{t-s}$ for some 
$0\leq r_{s+1}\leq t-s$.
Then $\sum_{i=1}^{s+1} r_i=r$ and $\dim(A)=\sum_{i=1}^{s} r_im _i+r_{s+1}$.
Since~$A$ is an optimal $r$-anticode, \cref{T-OptAnti}(a) and \cref{L-Maximize}(b) tell us that
\[
   r_i=n_i\text{ for }i<\ell,\ r_i=0\text{ for }i>\ell'\text{ and }  \sum_{i=\ell}^{\ell'}r_i=\sum_{i=\ell}^{j-1}n_i+\delta,
\]
as desired. In particular, $A^{\sf h}=\{0\}$.

(b) By assumption $m_i=n_i=1$ for $i\geq \ell$. Hence $r=\sum_{i=1}^{j-1}n_i=\sum_{i=1}^{\ell-1}n_i+j-\ell$.
``$\supseteq$'' A space as in the set on the right hand side is clearly an~$r$-anticode.
Its dimension is $\sum_{i=1}^{\ell-1}n_im_i+j-\ell$. Since $j-\ell=\sum_{i=\ell}^{j-1}n_i$, \cref{L-Maximize} implies optimality of the anticode.
``$\subseteq$'' Let $A\in\mA_r(\Pi)$. 
In this case,  \cref{T-OptAntiPlain} implies $A=A_1\oplus\ldots\oplus A_{\ell-1}\times A^{\sf h}$, where~$A_i$ is an optimal $r_i$-anticode in 
$\F^{n_i\times m_i}$ for some $0\leq r_i\leq n_i$ and $A^{\sf h}$ is an optimal $r_{\ell}$-anticode in $\F^{t-\ell+1}$ for some 
$0\leq r_{\ell}\leq t-\ell+1$.
Then $\sum_{i=1}^{\ell} r_i=r$ and $\dim(A)=\sum_{i=1}^{\ell-1} r_im _i+r_{\ell}$. 
Hence \cref{L-Maximize}(b) imply $r_i=n_i$ for $i<\ell$ and $r_\ell=j-\ell$, as desired.
\end{proof}

\section{Nonlinear Sum-Rank Metric Anticodes}
\label{sec:non}

In this section we consider anticodes in the sum-rank metric that are not necessarily linear. 
Since we have to measure the size of such anticodes by their cardinality,  we restrict ourselves to codes over \emph{finite} fields.
Thus from now on let $\F=\F_q$ be a finite field of order~$q$.

In classical coding theory, a strong motivation for studying anticodes comes from the  Code-Anticode Bound for distance-regular graphs. We prove a version of this result for linear spaces endowed with translation-invariant metrics.

\begin{lemma}[\textbf{Code-Anticode Bound}]\label{L:codeanticode}
Let $V$ be a vector space over $\F$ and let $\mbox{dist}: V \times V \to \R$
be a translation-invariant distance function. Let 
$d$ be a positive integer and let $C,A \subseteq V$ be subsets with the following properties: $\mbox{dist}(X,Y) \ge d$ for all distinct $X,Y \in C$ and  $\mbox{dist}(X,Y) \le d-1$ for all $X,Y \in A$.
Then $|C||A| \leq |V|$.
\end{lemma}	

\begin{proof}
The lemma is trivial if $|A|=1$ or $|C|=1$. We henceforth assume $|A|, |C| \ge 2$.
For each $X \in V$ define $X+A:=\{ X + X' \mid X' \in A\}$. Then for any choice of distinct $X,Y \in C$ the sets  $X+A$ and $Y+A$ are disjoint.
Indeed,  for any $X',Y' \in A$ we have $X+X' = Y +Y'$  if and only if $Y-X = X'-Y'$.
Translation invariance of the distance implies ${\rm dist}(X,Y)={\rm dist}(0,Y-X)={\rm dist}(0,X'-Y')={\rm dist}(Y',X') \leq d-1$, which contradicts our assumption on $C$.
It follows that $|C||A| = |\bigcup_{X \in \mC} (X+A)|\leq |V|.$	
\end{proof}

A set $C$ as in the previous lemma is called a $d$-code. In the sum-rank metric, this specializes to the following notion.

\begin{definition}
Let $1 \le d \le N$ be an integer. A (\textbf{sum-rank metric}) \textbf{$d$-code} is a non-empty subset $C \subseteq \Pi$ such that
$\srk(X-Y) \ge d$ for all distinct $X,Y \in C$. We say that $C$ is \textbf{linear} if it is an $\F$-linear subspace of $\Pi$. In that case we write $C \le \Pi$. When $|C| \ge 2$, we let the \textbf{sum-rank distance} of $C$ be the integer
$\srk(C):=\min\{\srk(X-Y) \mid X,Y \in C, \, X \neq Y\}$.
\end{definition}

The Code-Anticode Bound is most commonly stated for distance-regular graphs, where 
a code is a subset of the vertices and the (\textbf{geodesic}) \textbf{distance} between a pair of vertices is the minimum length of a path joining them. In that context, the Code-Anticode Bound is a special case of a result of Delsarte on association schemes; 
see~\cite{ahlswede+,delsarte}. 
The graph $\Gamma$ is generally constructed from a finite metric space $(V,\mbox{dist})$ taking as $V$ the set of vertices, declaring $v,w$ adjacent if $\mbox{dist}(v,w)=1$, and checking that for all $v,w \in V$ the geodesic distance between $v$ and $w$ coincides with 
$\mbox{dist}(v,w)$.
We briefly consider this graph for the sum-rank metric.

\begin{proposition}\label{P-Geod}
Define $\Gamma(\Pi)$ to be the graph whose vertices are the distinct elements of $\Pi$ and whose edges are the 
pairs $(X,Y) \in \Pi$ satisfying $\srk(X-Y)=1$. Then the geodesic distance in~$\Gamma(\Pi)$ coincides with the sum-rank distance.
\end{proposition}

\begin{proof}
Let $X=(X_1,\ldots,X_t),\,Y=(Y_1,\ldots,Y_t) \in \Pi$ be such that 
$\rk(X_i-Y_i)=r_i$ and $r=\sum_{i=1}^t r_i$, thus $\srk(X-Y)=r$. 
Then there exist~$r_i$ (and no fewer) matrices $A_{ij}$ of rank 1 satisfying $X_i - Y_i=\sum_{j=1}^{r_i} A_{ij}$. 
It follows that
$X-Y$ can be written as a sum of~$r$ matrices in $\Pi$ of sum-rank 1, and~$r$ is the minimum number of sum-rank-1
matrices needed. 
\end{proof}

It is interesting to note that our version of the Code-Anticode Bound does not rely on the distance regularity of the underlying graph.
Indeed, $\Gamma(\Pi)$ is not distance regular, as the next example shows. 
Recall that a graph $\Gamma$ with geodesic distance $\gamma : \Gamma \times \Gamma \to \N_0$ is \textbf{distance-regular} if 
\[
   \big|\{v\in\Gamma\mid \gamma(v,a)=i,\,\gamma(v,b)=j\}\big| \; \text{ depends only on $i,j$, and $\gamma(a,b)$}.
\]

\begin{example}
Let $\Pi=\F_2^{2 \times 2} \oplus \F_2^{2 \times 2}$ and let $A,B \in \Pi$ be the matrices 
\[
    A = \left(\begin{pmatrix} 1 & 0 \\ 0 & 1 \end{pmatrix}, \begin{pmatrix} 0 & 0 \\ 0 & 0 \end{pmatrix} \right), \quad  
	B = \left( \begin{pmatrix} 1 & 0 \\ 0 & 0 \end{pmatrix}, \begin{pmatrix} 1 & 0 \\ 0 & 0 \end{pmatrix}\right).
\]
Then $\srk(A)=\srk(B)=2$, and one can check that $\big|\{ Z \in \Pi \mid \srk(Z-A)=2, \, \srk(Z)=1\}\big| = 3$, whereas
$\big|\{ Z \in \Pi \mid \srk(Z-B)=2, \, \srk(Z)=1\}\big|=8$.
\end{example} 

The Singleton Bound and the Sphere-Packing Bound derived in~\cite[Theorems 3.2 and~3.6]{us_srk} are special instances of the 
Code-Anticode Bound of Lemma~\ref{L:codeanticode}. 
We first present these bounds. 

\begin{theorem}[\text{\cite[Theorems 3.2 and~3.6]{us_srk}}] \label{twot}
Let $C \subseteq \Pi$ be a code with $|C| \ge 2$ and $\srk(C) = d$.
\begin{enumerate}
\item (\textbf{Singleton Bound}) Let $j$ and $\delta$ be the unique integers such that $d-1=\sum_{i=1}^{j-1}n_i+\delta$ and $0\leq\delta\leq n_j-1$.
Then
\[
    |C|\leq q^{ \sum_{i=j}^t m_in_i- m_j\delta}.
\]
Codes meeting the Singleton Bound are called \textbf{MSRD}.
\item (\textbf{Sphere-Packing Bound}) Let $r=\lfloor(d-1)/2\rfloor$. Then 
\[
       |C| \leq \bigg\lfloor\frac{|\Pi|}{V_r(\Pi)}\bigg\rfloor,
\]
where 
\begin{equation} \label{V}
	V_r(\Pi):=\sum_{s=0}^r  \ \sum_{(s_1,...,s_t)\in\mU_s} \ 
	\prod_{i=1}^t \qbin{n_i}{s_i}{q} \prod_{j=0}^{s_i-1} (q^{m_i} - q^j)
\end{equation}
is the size of any sphere in~$\Pi$ of sum-rank radius~$r$~\cite[Lem. 3.5]{us_srk}. Codes meeting this bound are called \textbf{perfect}.
\end{enumerate}
\end{theorem}

Clearly, the Sphere-Packing Bound is an instance of the Code-Anticode Bound because the ball of radius $\lfloor (d-1)/2 \rfloor$ 
gives the required $(d-1)$-anticode in Lemma~\ref{L:codeanticode}.  

\begin{proposition}
The Singleton Bound is an instance of the Code-Anticode Bound of Lemma~\ref{L:codeanticode}.
\end{proposition}

\begin{proof}
Recall the notation from~\eqref{e-Pi} and the sets $\mU_r$ from \cref{T-OptAnti}.
For each $u \in \mathcal{U}_{d-1}$, define the projection
\[
      \tau_u : \Pi \longrightarrow \Pi\big((n_1-u_1)\times m_1 ,...,(n_t-u_t)\times m_t\big), \qquad X \longmapsto (X'_1,...,X'_t),
\]
where $X'_i$ is obtained from 
$X_i \in \F^{n_i \times m_i}$ by deleting its last $u_i$ rows. The kernel of this map is a (linear) $(d-1)$-anticode $A(u)$ and so by Lemma \ref{L:codeanticode}, we have $|\mC||A(u)|\leq |\Pi|$ for any code $C \subseteq \Pi$ with $\srk(C)=d$ and any choice of $u \in \mathcal{U}_{d-1}$. The Singleton Bound is obtained by choosing $u$ such that $|A(u)|=q^{\sum_{i=1}^tu_im_i}$ is maximal. 
By \cref{L-Maximize}(a) this is the case for
$u=(n_1,...,n_{j-1},\delta,0,...,0)$ for $j\in [t]$ and $0 \leq \delta <n_j$ as in \cref{twot}(a).  
\end{proof}

Note that the maximal linear $(d-1)$-anticode appearing in the proof of the above result is the linear space 
\[
    \mA(n_1,...,n_{j-1},\delta,0,...,0) = \Big( \mediumoplus_{i=1}^{j-1} \F^{n_i \times m_i} \Big) \oplus \F^{\delta \times m_j}.
\]

The Sphere-Packing Bound is sharper than the Singleton Bound precisely for those parameters for which the size of an optimal linear $(d-1)$-anticode is less than the size of the sphere of radius $\lfloor (d-1)/2 \rfloor$. In the remainder of the paper we compare the sizes of these two anticodes. 
We first show that for sufficiently large field size~$q$ an optimal linear $r$-anticode is larger than the ball of radius $\lfloor r/2 \rfloor$.

\begin{proposition}
Let $1 \le r \le N$ be an integer.
For sufficiently large~$q$, the cardinality of the sum-rank sphere of radius $\lfloor r/2 \rfloor$ in $\Pi$ is exceeded by the size of an optimal linear $r$-anticode. 
\end{proposition}

\begin{proof}
The result is clear if $r=1$. We henceforth assume $r \ge 2$ and denote by~$K$ the dimension of an optimal linear 
$r$-anticode in $\Pi$. 
Lemma \ref{L-Maximize} tells us that
\begin{equation}\label{e-MaxDim}
    K=\max\!\big\{\sum_{i=1}^{t} u_i m_i\,\big|\, u\in\mU_r\big\}\ \text{ and }\ K>
    \sum_{i=1}^{t} u_i m_i \text{ for all }u\in \bigcup_{a=0}^{r-1}\mU_a.
\end{equation}
Let us now turn to the ball. 
We can write the size of the ball of sum-rank radius $\rho:=\lfloor r/2 \rfloor$ in \eqref{V} as
\[
   V_{\rho}(\Pi)
     =\sum_{s=0}^{\rho}  \ \sum_{(s_1,...,s_t)\in\mU_s} S(s_1,...,s_t),
\]
where $S(s_1,...,s_t)=\prod_{i=1}^t \Gaussian{n_i}{s_i}_q\prod_{j=0}^{s_i-1} (q^{m_i} - q^j)$, which is the cardinality of the set 
of tuples $(X_1,\ldots,X_t)$ in $\Pi$ satisfying $\rk(X_i) = s_i$ for each $i \in [t]$.
It is straightforward to check that $S(s_1,...,s_t) \sim q^{\sum_{i=1}^{t} s_i(m_i+n_i-s_i)}$ as $q \to\infty$.
Therefore, $V_{\rho}(\Pi) \in \mathcal{O}\left(q^{\ell(\rho)}\right)$ in Bachmann-Landau notation, where 
\[
   \ell(\rho)  = \max \bigg\{\sum_{i=1}^t s_i(m_i+n_i-s_i)  \, \bigg| \,  (s_1,...,s_t) \in {\mathcal U}_s, \, 0 \leq s \leq  \rho \bigg\}.
\]
Since $r\geq2$ by assumption, the maximum is clearly attained for some $s\in\{1,\ldots,\rho\}$.
Fix now  any $s\in\{1,\ldots,\rho\}$ and $(s_1,...,s_t) \in \mU_s$.
Set $r_i=\min\{n_i,2s_i\}$ for $i\in[t]$. Then 
\begin{equation}\label{e-Rsum}
   \sum_{i=1}^t r_i\leq\sum_{i=1}^t 2s_i\leq 2s\leq r.
\end{equation}
Hence $(r_1,\ldots,r_t)\in\mU_a$ for some $a\leq r$. Furthermore we have
\begin{equation}\label{e-smn}
     \left.\begin{array}{ccccl}
     s_i(m_i+n_i-s_i)&\leq& n_im_i& \text{ and }& [s_i(m_i+n_i-s_i)= n_im_i\Longleftrightarrow s_i=n_i],\\[.5ex]
     s_i(m_i+n_i-s_i)&\leq &2s_im_i&\text{ and }&[s_i(m_i+n_i-s_i)= 2s_im_i\Longleftrightarrow s_i=0].
     \end{array}\quad\right\}
\end{equation}
The first inequality follows from maximizing the function $f(x)=x(m_i+n_i-x)$ 
on the interval $[0,n_i]$, while recalling $m_i\geq n_i$, the second inequality is clear.
Note that since $\sum_{i=1}^t s_i=s\leq\rho<N$, there is at least one~$i$ such that $s_i<n_i$. 
\\
\underline{Case 1:} Suppose $0<s_{\lambda}<n_{\lambda}$ for at least one~$\lambda$. 
Then  $s_{\lambda}(m_{\lambda}+n_{\lambda}-s_{\lambda})<r_{\lambda}m_{\lambda}$, and
\eqref{e-MaxDim} --~\eqref{e-smn} imply
\[
   \sum_{i=1}^ts_i(m_i+n_i-s_i)<\sum_{i=1}^t r_im_i\leq K.
\]
\underline{Case 2:} Suppose $s_i\in\{0,n_i\}$ for all $i\in[t]$. Set $\mI=\{i\mid s_i=n_i\}$. Then $\sum_{i\in\mI}n_i=s\leq\rho$ and
\[
     \sum_{i=1}^ts_i(m_i+n_i-s_i)=\sum_{i\in\mI}n_im_i< K
\]
thanks to~\eqref{e-MaxDim}.
Hence we showed that in either case the ball is strictly smaller than the optimal linear anticode.
\end{proof}

The previous result implies that for sufficiently large~$q$, the Code-Anticode Bound is sharper when the chosen anticode is an optimal linear anticode, rather than the sum-rank ball of the appropriate radius.  For small values of $q$, this does not necessarily hold, as the following example shows.

\begin{example}\begin{enumerate}
\item Let $\Pi =\mediumoplus_{i=1}^3 \F_2^{2 \times 2}$. Then 
	$V_2(\Pi) = 289$, while the size of an optimal linear $4$-anticode in $\Pi$
	is 256.
\item Now let $\Pi = \mediumoplus_{i=1}^3\F_3^{2 \times 2}$. Then $V_2(\Pi) = 3313$, while
	the size of an optimal linear $4$-anticode in $\Pi$ is 6561.
\item Finally, let $\Pi = \mediumoplus_{i=1}^4\F_2^{2 \times 2}$. 
	Then $V_2(\Pi)=547$, while the size of an optimal linear $4$-anticode in $\Pi$ is 256. 
\end{enumerate}
\end{example}

The Diametric Theorem of Ahlswede et al.~\cite{ahlswedekhachatrian} is reformulated in \cite[Theorem AK]{ahlswede+} 
in order to show that the size of an optimal $r$-anticode in Hamming space equals the cardinality of the Cartesian product 
of the Hamming ball of radius $r$ in $\F_q^{n-\ell+2r}$ and the optimal linear $(\ell-2r)$-anticode $\oplus_{j=1}^{\ell-2r} \F_q$ 
for a suitable~$\ell$.
For the sum-rank metric the following example suggests a similar behavior: the direct product of a suitable ball and full matrix spaces leads to a nonlinear anticode with larger cardinality than both the ball and the optimal linear anticode.  

\begin{example}
Let $\F=\F_3$ and $\Pi=\mediumoplus_{i=1}^7\F^{2 \times 2}$ and $r=8$. 
Denote by $B_s(k)$ the ball of radius~$s$ in $\mediumoplus_{i=1}^k\F^{2 \times 2}$. 
We have the following $8$-anticodes, with $A_0$ being the optimal linear $8$-anticode and $A_4$ being the ball of radius~$4$:
\[
   \begin{array}{ll}
   A_0=\mediumoplus_{i=1}^4\F^{2 \times 2},\ & |A_0|=43,046,721,\\[1ex]
   A_1=B_1(4)\oplus\mediumoplus_{i=1}^3\F^{2 \times 2},\qquad\qquad & |A_1|=68,555,889,\\[1ex]
    A_2=B_2(5)\oplus\mediumoplus_{i=1}^2\F^{2 \times 2},\ & |A_2|=69,815,602, \\[1ex]
   A_3=B_3(6)\oplus\F^{2 \times 2}, & |A_3|=58,099,761,\\[1ex]
    A_4=B_4(7), &  |A_4|=43,142,961.     \end{array}
\]
Thus the ``hybrid'' anticodes $A_1,\,A_2,\,A_3$ are strictly larger than the optimal linear anticode and the ball of radius~$4$, and~$A_2$ is largest among those considered.
We also note that over the field~$\F_2$, the ball $A_4$ is largest,  over $\F_4$ the anticode $A_1$ is largest, and over~$\F_5$ 
the linear anticode $A_0$ is largest.
\end{example}

The explicit description of optimal (nonlinear) anticodes in the sum-rank metric appears to be an interesting open problem. We leave this to future research.

\bigskip

\bigskip

\bibliographystyle{abbrv}

\end{document}